\documentclass[11pt,reqno]{amsart}
\usepackage{amsthm}
\usepackage{amsmath,amssymb,txfonts,tipa,amsfonts,ascmac,mathrsfs,multicol,amscd,ascmac,fancybox}
\usepackage[dvipdfmx]{graphicx}
\usepackage[pagewise]{lineno}
\usepackage{amsaddr}
\usepackage[margin=1.2in]{geometry}

\usepackage{url}
\usepackage[initials]{amsrefs}
\usepackage{amsmath,amsthm,amscd}
\usepackage{enumerate}

\usepackage{amssymb}
\usepackage{mathrsfs}
\usepackage{graphicx}

\usepackage{mathtools}

\makeatletter
\@namedef{subjclassname@2010}{%
  \textup{2010} Mathematics Subject Classification}
\makeatother

\usepackage{hyperref}
\hypersetup{
setpagesize=false,
 bookmarksnumbered=true,
 bookmarksopen=true,
 colorlinks=true,
 linkcolor=blue,
 citecolor=red,
}

\numberwithin{equation}{section}

\newtheorem{theorem}{Theorem}[section]

\newtheorem{proposition}[theorem]{Proposition}
\newtheorem{corollary}[theorem]{Corollary}

\theoremstyle{definition}
\newtheorem{definition}[theorem]{Definition}
\newtheorem{remark}[theorem]{Remark}
\newtheorem{problem}[theorem]{Problem}

\begin{document}

\baselineskip=17pt

\title[Mixing and observation for Markov operator cocycles]{Mixing and observation for Markov operator cocycles}

\author[F. Nakamura]{Fumihiko NAKAMURA}
\address[F. Nakamura]{Faculty of Engineering, Kitami Institute of Technology, Hokkaido, 090-8507, JAPAN}
\email[F. Nakamura]{nfumihiko@mail.kitami-it.ac.jp}

\author[Y. Nakano]{Yushi NAKANO}
\address[Y. Nakano]{Department of Mathematics, Tokai University, Kanagawa 259-1292, JAPAN}
\email[Y. Nakano]{yushi.nakano@tsc.u-tokai.ac.jp}

\author[H. Toyokawa]{Hisayoshi TOYOKAWA}
\address[H. Toyokawa]{Institute of Mathematics for Industry, Kyushu University, Fukuoka, 819-0395, JAPAN}
\email[H. Toyokawa]{toyokawa@imi.kyushu-u.ac.jp}

\begin{abstract}
We consider generalized definitions of mixing and exactness for random dynamical systems in terms of Markov operator cocycles. We first give six fundamental  definitions of mixing for Markov operator cocycles in view of  observations of the randomness in environments, and reduce them into two  different groups. Secondly, we give the definition of exactness for Markov operator cocycles and show that Lin's criterion for exactness can be naturally extended to the case of Markov operator cocycles. Finally, in the class of asymptotically periodic  Markov operator cocycles, we show the Lasota-Mackey type equivalence between mixing, exactness and asymptotic stability. 
\end{abstract}


\maketitle

\section{Introduction}

  This paper concerns a cocycle generated by Markov operators, called a {\it Markov operator cocycle}. 
  Let $(X,\mathcal{A},m)$ be a probability space, and  $L^1(X, m)$  (the quotient   by equality $m$-almost everywhere of) the space of all $m$-integrable functions on $X$, endowed with  the usual $L^1$-norm  $\Vert \cdot \Vert _{L^1(X)}$.
An operator $P: L^1(X, m)\to L^1 (X, m)$ is called a \emph{Markov operator} if $P$ is linear, positive (i.e.~$Pf\geq 0$ $m$-almost everywhere if $f\geq 0$ $m$-almost everywhere) and
 \begin{equation}\label{eq:0219}
\int _X Pf  dm  = \int _Xf dm \quad
\text{ for all $f \in L^1(X, m)$} .
\end{equation}
  Markov operators  naturally appear in the study of dynamical systems (as Perron-Frobenius operators; see \eqref{eq:0914b}), Markov processes (as integral operators with the stochastic kernels of the processes), and random dynamical systems in the annealed regime (as integrations of Perron-Frobenius operators over environmental parameters).
 For these  deterministic/stochastic dynamics, 
  $\{ P^nf\} _{n\geq 0}$  is the evolution of density functions of random variables driven by the system.
  We refer to \cites{F,LM}. 
  
A  Markov operator cocycle is given by compositions of different  Markov operators which are provided with according to  the environment $\{ \sigma ^n(\omega )\} _{n\geq 0}$ driven by a measure-preserving  transformation $\sigma :\Omega \to \Omega $ on a probability space $(\Omega , \mathcal F, \mathbb P)$, 
\[
\mathbb N \times \Omega \times L^1(X,m) \to  L^1(X,m):  (n, \omega , f) \mapsto P_{\sigma ^{n-1}(\omega )} \circ P_{\sigma ^{n-2}(\omega )} \circ \cdots \circ P_{\omega } f
\]
 (see Definition \ref{dfn:1} for more precise description). 
So, in  nature it
possess two kinds of randomness:
\begin{itemize}
\item[(i)] The evolution of densities at each time are dominated by Markov operators $P_\omega$, 
\item[(ii)] The selection of  each Markov operators  is driven by the base dynamics $\sigma$.
\end{itemize}
The aim of this paper is to investigate how the \emph{observation} of the randomness of the state space and the environment influences  statistical properties of the system, and to give a step to understanding more complicated phenomenon in multi-stochastic systems.

Our focus lies on the mixing property.
Recall that a Markov operator $P: L^1(X,m) \to L^1(X,m)$ is said to be \emph{mixing} if 
\begin{equation}\label{eq:0912}
\int _X P^nf g dm \to \int _Xfdm \int _X gdm \quad \text{as $n\to \infty$}
\end{equation}
for  any $f\in L^1(X,m)$ and $g\in L^\infty (X,m)$ (when $P1_X =1_X$, see Remark \ref{B} for more general form).
Due to \eqref{eq:0219}, this means that two random variables $P^nf$ and $g$ are asymptotically independent so that  the system is considered to ``mix'' the state space well.
In other words, the randomness of $P$ in the sense of mixing  can be seen through  the  observables $f$ and $g$.
Hence,  for Markov operator cocycles,  the strength of the dependence  of the observables on $\omega$ expresses how one observes  the randomness of the state space and the environment. 
Furthermore, more directly, we can consider different kinds of mixing properties according to whether   the environment $\omega$ is observed as \emph{a prior  event} to the observation of $f ,g$.
According to these viewpoints, we will introduce six definitions of mixing for Markov operator cocycles (Definition \ref{eme}).
In Section \ref{s:mixing}, we show that five of them are equivalent when $\Omega$ is a compact topological space, while at least one of them are different. In the case when the Markov operator cocycle is generated by a random dynamical system over a mixing driving system, we also show that all of them imply the (conventional) mixing property of the skew-product transformation induced by the random dynamical system.

We further investigate exactness for  Markov operator cocycles. 
 Since  the observable $g$ in \eqref{eq:0912} does not appear in the definition of exactness for a Markov operator $P$ (recall that, when $P1_X=1_X$, $P$ is said to be exact if  $\displaystyle\lim _{n\to \infty}$ $\left\lVert P^nf  -\int _X fdm\right\rVert _{L^1 (X)} =0$ for all $f\in L^1(X,m)$; see also the remark following Definition \ref{d:32}), in contrast to the mixing property, we only have one  definition of exactness for  Markov operator cocycles (Definition \ref{d:32}).
We will show that Lin's criterion for exactness can be naturally extended to the case of Markov operator cocycles (Section \ref{s:exact}),
and finally, in the class of asymptotically periodic Markov operator cocycles, we 
 prove Lasota-Mackey type equivalence between mixing, exactness and asymptotic stability, as well as their relationship with the existence of an invariant density map
   (Section \ref{s:asymp}).
See Figure \ref{fig:FC} for the summary.

\begin{figure}[ht]
\begin{center}
\includegraphics[bb=70 120 800 500, width=15cm, angle=90]{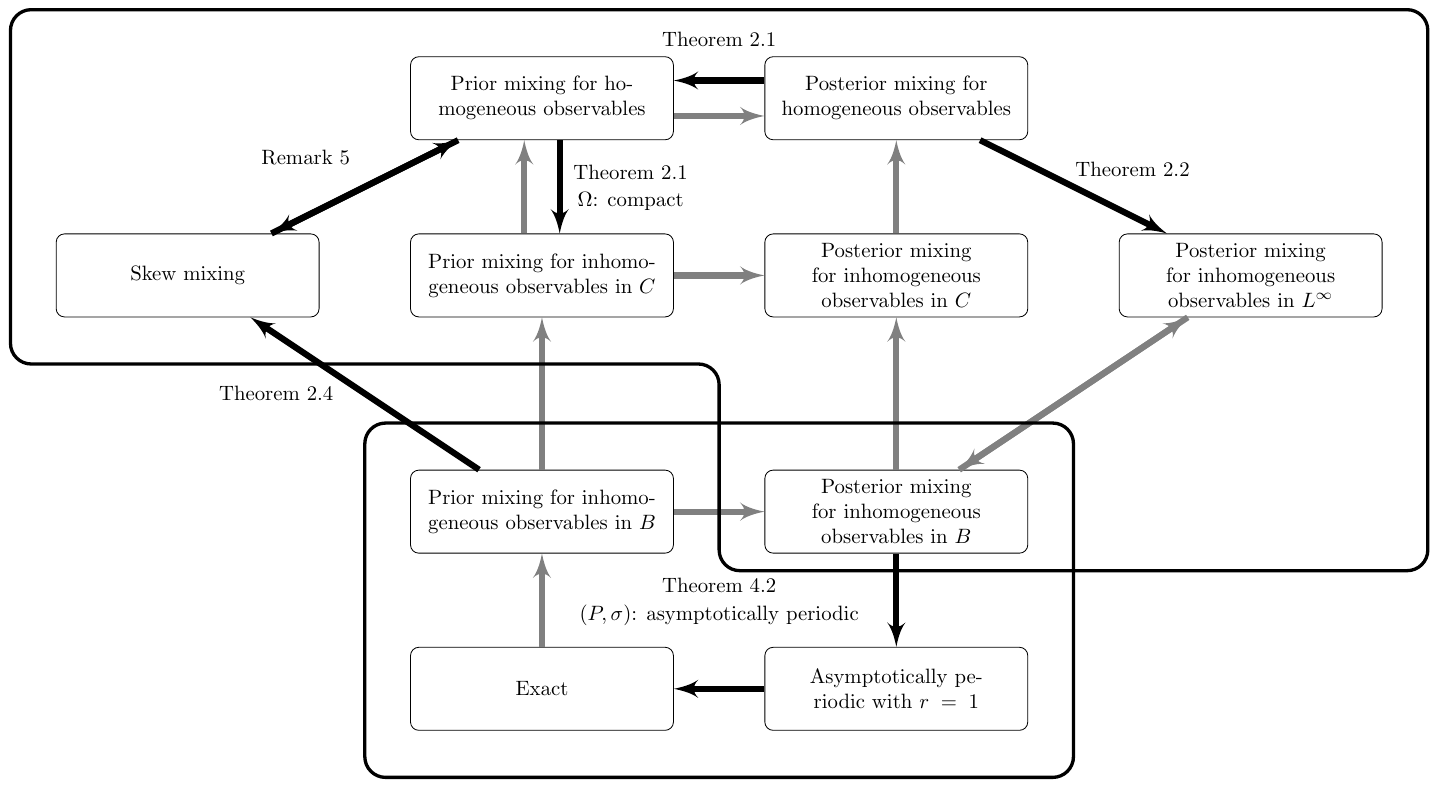}
\caption{
The relations between definitions in this paper. 
Here $B$, $C$ and $L^\infty$ are abbreviations of $B(\Omega, L^\infty (X,m))$, $C(\Omega, L^\infty (X,m))$ and  $L^\infty(\Omega, L^\infty (X,m))$, respectively.
The implications by a gray arrow represent trivial relations by definitions. 
The implication by a black arrow from prior mixing for homogeneous observables  to prior mixing for inhomogeneous observables in $C$ holds when  $\Omega$ is a compact topological space. Therefore,
 the above seven definitions are all equivalent when $\Omega$ is  compact. Moreover, the below four definitions are equivalent for an asymptotically periodic Markov operator cocycle, while it is not true for general Markov operator cocycles due to Corollary \ref{cor:0903}.
} 
\label{fig:FC}
\end{center}
\end{figure}

\subsection{Definitions of mixing and exactness}

Let $D(X, m)$ and $L^1_0(X,m)$ be subsets of $L^1 (X, m)$ given by
\begin{eqnarray}
 D(X, m) &=&\left\{f\in L^1(X, m) : f\geq 0 \  \text{$ m$-almost everywhere, $\Vert f\Vert _{L^1(X)}=1$} \right\},\nonumber\\
 L^1_0(X, m) &=&\left\{f\in L^1(X, m) : \int_X fdm = 0 \right\}. \nonumber 
\end{eqnarray}
Note that $P:L^1(X,m)\to L^1(X,m)$ is a Markov operator  if and only if  $P(D(X,m))$ $\subset D(X,m)$.

One of the most important examples of Markov operators is the \emph{Perron-Frobenius operator} induced by a non-singular transformation $T:X\to X$ (that is, $T_* m $ is absolutely continuous with respect to $m$, where $T_*m$ is the pushforward of $m$  given by $T_*m(A) =m(T^{-1} A)$ for $A\in \mathcal A$).
The Perron-Frobenius operator $L_T: L^1(X,m) \to L^1(X,m)$ of $T$ is defined by
\begin{equation}\label{eq:0914b}
L_T f = \frac{d [T_*(f m)]}{dm} \quad \text{for $f\in L^1(X,m)$},
\end{equation}
where $fm$ is a finite signed measure given by $(fm)(A) =\int _A f dm$ for $A\in \mathcal A$ and $d\mu /dm$ is the Radon-Nikodym derivative of an absolutely continuous finite signed measure $\mu$.
Note that for each $X$-valued random variable $\chi$ whose distribution is $fm$ with some $f\in D(X,m)$, $T(\chi )$ has the distribution  $(L_T f) m$ (and thus, $L_T$ is also called the transfer operator associated with  $T$).
It is straightforward to see that
\begin{equation}\label{eq:0913}
\int_X L_T fg dm = \int_X fg\circ T dm \quad\quad \text{for $f\in L^1(X,m)$  and $g\in L^{\infty}(X,m)$},
\end{equation}
and that  $L_T$ is a Markov operator.

Recall that $(\Omega, \mathcal F , \mathbb{P})$ is a probability space, and $\sigma:\Omega\to\Omega$ is a $\mathbb P$-preserving transformation.
For a measurable space $\Sigma$,  we say that a measurable map $\Phi: \mathbb N_0 \times \Omega \times \Sigma \to \Sigma$ is a \emph{random dynamical system} on $\Sigma$ over the driving system $\sigma$ if
\[
\varphi ^{(0)} _\omega = \mathrm{id} _{\Sigma}  \quad\text{and}\quad \varphi ^{(n+m)} _ \omega  = \varphi ^{(n)}_{ \sigma ^m\omega }\circ  \varphi ^{(m)}_\omega
\]
for each $n, m \in \mathbb N_0$ and $\omega \in \Omega$, with the notation $\varphi ^{(n)}_\omega =\Phi (n,\omega ,\cdot )$ and $\sigma \omega =\sigma (\omega )$, where $\mathbb N_0 =\mathbb N \cup \{0\}$. 
A standard reference for random dynamical systems is the monographs by Arnold \cite{Ar}. It is easy to check that 
\begin{equation}\label{eq:0220b2}
\varphi ^{(n)}_\omega = \varphi _{\sigma ^{n-1}\omega }\circ \varphi _{\sigma ^{n-2}\omega } \circ \cdots \circ \varphi _\omega 
\end{equation}
with the notation $\varphi _\omega = \Phi (1, \omega , \cdot )$.
Conversely, for each measurable map $\varphi : \Omega \times \Sigma\to \Sigma: (\omega , x) \mapsto \varphi _\omega (x)$, the measurable map $(n,\omega , x) \mapsto \varphi _\omega ^{(n)}(x)$ given by \eqref{eq:0220b2} is a random dynamical system. 
We call it a random dynamical system induced by $\varphi$ over $\sigma$, and simply denote it by $(\varphi , \sigma )$.
When $\Sigma$ is a Banach space and $\varphi _\omega :\Sigma \to \Sigma$ is $\mathbb P$-almost surely linear, $(\varphi , \sigma )$ is called a \emph{linear operator cocycle}. 
We give a formulation of Markov operators in random environments in terms of linear operator cocycles.

\begin{definition}\label{dfn:1}
We say that a linear operator cocycle $(P, \sigma )$  induced by a measurable map $P: \Omega \times L^1(X,m) \to L^1(X,m)$ over $\sigma$ is a  \emph{Markov operator cocycle} (or a \emph{Markov operator in random environments}) if $P_\omega =P(\omega ,\cdot ): L^1 (X,m) \to L^1 (X,m)$ is a Markov operator for $\mathbb P$-almost every $\omega \in \Omega$.
\end{definition}
 
Let $(n, \omega , f) \mapsto P^{(n)}_\omega f$ be a Markov operator cocycle induced by $P: \Omega \times L^1(X,m) \to L^1(X,m)$ such that $P_\omega =P(\omega ,\cdot )$  is the Perron-Frobenius operator $L_{T_\omega}$ associated with  a   non-singular map $T_\omega :X\to X$ for  $\mathbb P$-almost every $\omega$.
 Then, it follows from \eqref{eq:0913} that $\mathbb P$-almost surely
 \begin{align}\label{eq:0913b}
 \int_X P_{\omega}^{(n)}f g dm = \int_X f g\circ T_\omega^{(n)} dm, \quad \text{for }f\in L^1(X,m) \text{ and } g\in L^{\infty}(X,m),
 \end{align}
 where $T_\omega^{(n)} = T_{\sigma ^{n-1}\omega }\circ T_{\sigma ^{n-2}\omega }\circ \cdots \circ T_\omega$.

We are now in place to give definitions of mixing  for Markov operator cocycles.
Let $K$ be a  space consisting of measurable maps from $\Omega$ to $L^\infty (X,m)$.

\begin{definition}\label{eme}
A Markov operator cocycle.
 $(P,\sigma)$ is called
 \begin{enumerate}
 \item {\it prior mixing for homogeneous observables} if 
 for $\mathbb{P}$-almost every $\omega\in\Omega$,  any $f\in L^1_0(X,m)$ and $g\in L^\infty (X,m)$, it holds that
  \begin{align}\label{hom}
  \lim_{n\to\infty}\int_X P^{(n)}_{\omega}f g dm =0;
 \end{align}
    \item {\it posterior mixing for homogeneous observables} if for any $f\in L^1_0(X,m)$, $g\in L^\infty (X,m)$ 
 and $\mathbb{P}$-almost every $\omega\in\Omega$,
\eqref{hom} holds;
 \item {\it prior mixing for inhomogeneous observables in $K$} if 
 for $\mathbb{P}$-almost every $\omega\in\Omega$,  any $f\in L^1_0(X,m)$ and $g\in K$, it holds that
  \begin{align}\label{ihom}
  \lim_{n\to\infty}\int_{X}P^{(n)}_{\omega}f g_{\sigma^n\omega} dm=0;
  \end{align}
 \item {\it posterior mixing for  inhomogeneous  observables in $K$} if for any $f\in L^1_0(X,m)$, $g\in K$ 
 and $\mathbb{P}$-almost every $\omega\in\Omega$,
\eqref{ihom} holds.
 \end{enumerate}
\end{definition}
In the prior  case (the posterior case), the observation of the environment $\omega$ is \emph{a prior event} (\emph{a posterior event}, respectively) to the observation of $f$ and $g$.
As the class of   inhomogeneous observables $K$ in Definition \ref{eme}, we will consider the following two fundamental  classes.
\begin{itemize}
\item[(i)]  $B(\Omega, L^\infty (X,m))$: the set of all bounded and measurable maps from $\Omega$ to $L^\infty (X,m)$.
\item[(ii)]  $C(\Omega, L^\infty (X,m))$: the set of all bounded and continuous maps from $\Omega$ to $L^\infty (X,m)$ (when $\Omega$ is a topological space and $\mathcal F$ is its Borel $\sigma$-field).
\end{itemize}

\begin{remark}\label{B}
The above definitions need not require an invariant density map for the Markov operator cocycle $(P,\sigma)$.
We say that a measurable map $h: \Omega \to D(X,m)$ is  an {\it invariant density map} for $(P,\sigma)$ if $P_{\omega}h_{\omega}=h_{\sigma\omega}$ holds for $\mathbb{P}$-almost every $\omega\in\Omega$ where $h_{\omega}=h(\omega)$.
Now we assume that there exist  an invariant density map  $h:\Omega \to D(X,m)$  for $(P,\sigma)$ such that for $\mathbb{P}$-almost every $\omega\in\Omega$,
 \begin{align}\label{A}
 \lim_{n\to\infty}m \left(\operatorname*{supp}P_{\omega}^{(n)}1_X\setminus \operatorname*{supp}P_{\omega}^{(n)}h_{\omega}\right)=0.
 \end{align}
Then by (\ref{A}) and the fact that $P^{(n)}_{\omega}f-h_{\sigma^{n}\omega}=P^{(n)}_{\omega}(f-h_{\omega})\in L^1_0(X,m)$ for $f\in D(X,m)$,
one can easily check that $(P,\sigma)$ is
 prior mixing for homogeneous observables if and only if for $\mathbb{P}$-almost every $\omega\in\Omega$,  any $f\in D(X,m)$ and $g\in L^{\infty}(X,m)$, it holds that
  \begin{align*}
  \lim_{n\to\infty}\int_X \left(P^{(n)}_{\omega}f-h_{\sigma^n\omega}\right)g dm = 0.
  \end{align*}
Furthermore, when $P_\omega$ is the Perron-Frobenius operator $L_{T_\omega}$ associated with  a  non-singular map $T_\omega : X\to X$, by \eqref{eq:0913b}, 
it is also equivalent to that  for $\mathbb{P}$-almost every $\omega\in\Omega$, any $f\in L^1(X,\mu _\omega )$ and $g\in L^{\infty}(X,m)$,
  \begin{align}\label{eq:0913c}
 \int_Xfg\circ T_\omega ^{(n)} d\mu _\omega  - \int _X fd\mu _\omega \int _X g d\mu _{\sigma^n\omega} \to 0 \quad \text{as $n\to \infty$},
  \end{align}
  where $\mu _\omega =h_\omega m$.
%
Moreover, we  can replace ``for any $f\in L^1(X,\mu _\omega)$'' in the previous sentence with ``for any measurable function $f:  \Omega \times X\to \mathbb R$ such that $f_\omega =f(\omega , \cdot )\in L^1(X,\mu _\omega)$  $\mathbb P$-almost surely'', and ``$f$'' in \eqref{eq:0913c} with ``$f_\omega$''.
 Similar equivalent conditions can be found for other types of mixing in Definition \ref{eme}.
 
 All kinds of mixing in Definition \ref{eme} were adopted in literature, especially  in the form of \eqref{eq:0913c} to  discuss  mixing for random dynamical systems.
 For instance, we refer to 
Baladi et al.~\cites{BKS,BY} and Buzzi~\cite{B} for the definition 1, 
Dragi\v{c}evi\'{c} et al.~\cite{DFGV} for the definition 2, 
Bahsoun et al.~\cite{BBR} for the definition 3, and 
Gundlach \cite{Gundlach1996} for the definition 4. 
Moreover, in the deterministic case (i.e.~$\Omega$ is a singleton), all the definitions  are equivalent to the usual definition of mixing  for a single Markov operator  \cite{LM}.
\end{remark}

\begin{remark}\label{rm:0903b}
Another natural candidate for the class of  inhomogeneous observable is the Bochner-Lebesgue space $L^\infty(\Omega, L^\infty (X,m))$, that is,  the Kolmogorov quotient (by equality $\mathbb P$-almost surely) of the space of all $\mathbb P$-essentially bounded and Bochner measurable maps from $\Omega$ to $L^\infty (X,m)$ (and \eqref{ihom} is interpreted as it holds under the usual identification between an equivalent class and a representative of the class). 
However, in the case $K=L^\infty(\Omega, L^\infty (X,m))$, the prior version 3 does not make sense because one can find an equivalent class $[g] \in L^\infty(\Omega, L^\infty (X,m))$ and maps $g_1, g_2\in [g]$ such that \eqref{ihom} holds for $g=g_1$ while \eqref{ihom} does not hold for $g=g_2$, see  Subsection \ref{ss:ce}.
On the other hand, the posterior version 4 makes sense for $K=L^\infty(\Omega, L^\infty (X,m))$, and indeed, its relationship with posterior mixing for homogeneous observables will be discussed in Theorem \ref{thm:0917b}.
\end{remark}

By the definitions, we immediately see that the prior mixing implies the posterior mixing (that is, $(1)\Rightarrow(2)$ and $(3)\Rightarrow(4)$ in Definition \ref{eme}). 
It is also obvious that the prior (posterior) mixing for inhomogeneous observables in $B(\Omega, L^\infty (X,m))$ or $C(\Omega, L^\infty (X,m))$ implies the prior (posterior, respectively) mixing for homogeneous observables. 
Refer to Figure \ref{fig:FC}.

We next define exactness for Markov operator cocycles.
\begin{definition}\label{d:32}
A Markov operator cocycle
 $(P,\sigma)$ is called {\it exact} if for $\mathbb{P}$-almost every $\omega\in\Omega$ and any $f\in L^1_0(X,m)$, it holds that
  \begin{align}\label{exact}
  \lim_{n\to\infty}\left\lVert P^{(n)}_{\omega}f \right\rVert_{L^1(X)}=0.
  \end{align}
  \end{definition}

As in Remark \ref{B}, we can easily see that the exactness of a Markov operator cocycle  $(P,\sigma)$ 
is equivalent to that
for $\mathbb{P}$-almost every $\omega\in\Omega$ and  any $f\in D(X,m)$,
  \begin{align*}
  \lim_{n\to\infty}\left\lVert P^{(n)}_{\omega}f-h_{\sigma^n\omega} \right\rVert_{L^1(X)}=0.
  \end{align*}
In Section \ref{s:exact}, we will see another equivalent condition of the exactness in the case when $(P,\sigma)$ is associated with a random dynamical system on $X$. 
The relationship between mixing, exactness and asymptotic stability will be also discussed in Section \ref{s:asymp}, see again Figure \ref{fig:FC} for a summary.

\section{Mixing}\label{s:mixing}

\subsection{Equivalence}
We show  the equivalence between prior/posterior mixing for homogeneous observables and prior/posterior mixing for inhomogeneous observables in $C(\Omega, L^\infty (X,m))$ when $\Omega$ is  a compact topological space.

\begin{theorem}\label{Thm:mixing}
Assume that $\Omega$ is a compact topological space.
Then, the followings are equivalent:
\begin{enumerate}
\item $(P,\sigma)$ is prior  mixing for homogeneous observables.
\item $(P,\sigma)$ is posterior  mixing for homogeneous observables.
\item $(P,\sigma)$ is prior  mixing for inhomogeneous observables in $C(\Omega,L^\infty(X,m))$.
\item $(P,\sigma)$ is posterior mixing for inhomogeneous observables in $C(\Omega,L^\infty(X,m))$.
\end{enumerate}
\end{theorem}
\begin{proof}
As mentioned,  the implications $(3)\Rightarrow(4)$ and
$(4)\Rightarrow(2)$  immediately follow from the definitions.
We show $(2)\Rightarrow(1)$. 
Assume that $(P,\sigma)$ is posterior mixing for homogeneous observables, that is, for any $f\in L_0^1(X,m)$ and $g\in L^\infty(X,m)$,  there is a  measurable set $\Omega _0(f,g)$ such that $\mathbb P(\Omega _0(f,g))=1$ and  \eqref{hom} holds for each $\omega \in \Omega _0(f,g)$.
By the simple function approximation with rational coefficients, we can find countable dense subsets $\{ f_k\} _{k\in \mathbb N}$ of $L_0^1(X,m)$ and  $\{ g_l\} _{l\in \mathbb N}$ of $L^\infty(X,m)$. 
Define a measurable set $\Omega _0$ by
\[
\Omega _0 = \bigcap _{k\in \mathbb N}\bigcap _{l\in \mathbb N}\Omega _0(f_k,g_l),
\]
then it is straightforward to see that $\mathbb P(\Omega _0) =1$ and \eqref{hom} holds for any $\omega\in \Omega _0 $, $f\in L_0^1(X,m)$ and $g\in L^\infty(X,m)$, i.e.~$(P,\sigma)$ is prior mixing for homogeneous observables.

We next  show $(1)\Rightarrow(3)$. 
Assume that $(P,\sigma)$ is prior mixing for homogeneous observables, that is, there is a  measurable set $\Omega _0$ with $\mathbb P(\Omega _0)=1$ such that  \eqref{hom} holds for any $\omega \in \Omega _0$,  $f\in L_0^1(X,m)$ and $g\in L^\infty(X,m)$.
Fix such an $\Omega _0$.
Fix also  $f\in L_0^1(X,m)$, $g\in C(\Omega , L^\infty(X,m))$ and $\epsilon >0$.
Then, since $\Omega$ is compact, we get finitely many functions $\{ g_i\} _{i=1}^I \subset L^\infty(X,m)$ such that, for any $\omega \in \Omega$ there is $1\leq i(\omega ) \leq I$ satisfying
\[
\Vert g_\omega - g_{i(\omega )}  \Vert _{L^\infty(X)} < \epsilon .
\]
(Note that $\{ \{ \omega \in \Omega : \Vert g_\omega -\tilde g\Vert _{L^\infty(X)} < \epsilon \} : \tilde g\in L^\infty(X) \}$ is an open covering of $\Omega$ by virtue of the continuity of $g$.)
For convenience, let $g_0 =1_X$.

We further fix $\omega \in \Omega_0$.
By applying \eqref{hom} to $g= g_i$ with $0\leq i\leq I$, one can find $N_i \equiv N_i(\omega ,f)\in \mathbb N$ such that
\[
\left\vert  \int _X P^{(n)}_{\omega}f g_i  dm \right\vert < \epsilon \quad \text{for all $n\geq N_i$}.
\]
Hence, for any  $n\geq \max _{0\leq i\leq I} N_i$,
\begin{align}\label{eq:0917}
\notag\left\vert  \int _X  P^{(n)}_{\omega}f g_{\sigma ^n \omega} dm \right\vert 
&\leq \left\vert  \int _X  P^{(n)}_{\omega}f g_{i(\sigma ^n \omega )} dm \right\vert 
+ \left\Vert g_{\sigma ^n \omega}  -  g_{i(\sigma ^n \omega )} \right\Vert _{L^\infty(X)}  \left\Vert P^{(n)}_{\omega}f \right\Vert _{L^1(X)}\\
&< (1+ \Vert f\Vert _{L^1(X)}  )\epsilon  .\nonumber
\end{align}
Since $\epsilon >0$ is arbitrary, we conclude
\[
\lim _{n\to \infty} \int _X  P^{(n)}_{\omega}f g_{\sigma ^n \omega} dm =0 \quad \text{for all $\omega \in \Omega _0$},
\]
which implies that $(P,\sigma)$ is prior  mixing for inhomogeneous  observables in $C(\Omega,L^\infty$ $(X,m))$.
This completes the proof.
\end{proof}

\begin{remark}\label{rm:0903}
As in the proof, the compactness of $\Omega$ in Theorem \ref{Thm:mixing} is only needed to show the implication of prior  mixing for inhomogeneous observables in $C(\Omega,L^\infty(X,$ $m))$ from prior  mixing for homogeneous observables.
\end{remark}

We also can show the following equivalence for observables in $L^\infty(\Omega,L^\infty(X,m))$.
\begin{theorem}\label{thm:0917b}
If $(P,\sigma)$ is posterior  mixing for homogeneous observables,
then $(P,\sigma)$ is posterior mixing for inhomogeneous observables in $L^\infty(\Omega,L^\infty(X,m))$.
\end{theorem}
\begin{proof}
Assume that $(P,\sigma)$ is posterior mixing for homogeneous observables, i.e.~for any $f\in L_0^1(X,m)$ and $g\in L^\infty(X,m)$,  there is a  measurable set $\Omega _0(f,g)$ such that $\mathbb P(\Omega _0(f,g))=1$ and  \eqref{hom} holds for each $\omega \in \Omega _0(f,g)$.
Fix $f\in L^1_0(X,m)$ and  $g\in L^\infty(\Omega,L^\infty(X,m))$.
We only consider the case when $g_\omega(x)$ is positive for $\mathbb P\times m$-almost every $(\omega ,x)\in \Omega \times X$. (If not, we consider the usual decomposition $g=g^+-g^-$ with $g^+_\omega (x) = \max \{ g_\omega (x) ,0\}$ and $g^-_\omega (x) = \max \{ - g_\omega (x) ,0\}$.)

Since $g\in L^\infty(\Omega,L^\infty(X,m))$ (in particular, $g$ is Bochner measurable), there is a sequence of  simple functions $\{ g^k\} _{k\in \mathbb N} \subset L^\infty(\Omega,L^\infty(X,m))$ of the form
\[
g^k_\omega (x)=\sum_{i=1}^{I(k)}  g_i^k(x) 1_{F_i^k}(\omega) \quad \left( g_i^k \in L^\infty(X,m),\ F_i^k\in\mathcal{F}\right)
\]
and a $\mathbb P$-full measure set $\Omega _1$
 such that 
$
\sup _{\omega \in \Omega _1} \left\Vert g_\omega - g_\omega ^k \right\Vert _{L^\infty (X)}\to 0$
as $k\to \infty$.
Define a $\mathbb P$-full measure set $\Omega _0$ by
\[
\Omega _0 =  \bigcap _{k\in \mathbb N}\bigcap _{1\leq i\leq I(k)}\Omega _0(f,g_i^k).
\]
Let $\Omega _2 = \Omega _0 \cap (\bigcap _{n\geq 0} \sigma ^{-n}\Omega _1)$, then $\mathbb P(\Omega _2)=1$ by the invariance of $\mathbb P$ for $\sigma$. 

Fix $\omega \in \Omega _2$ and $\epsilon >0$.
Fix also $k\in \mathbb N$ such that 
\[
\left\Vert g_{\sigma ^n\omega }- g_{\sigma ^n\omega} ^k \right\Vert _{L^\infty (X)} <\epsilon \quad \text{for all $n\in \mathbb N$}.
\]
Calculate that 
\begin{eqnarray}
\left\lvert \int_X P_\omega^{(n)}f g^k_{\sigma^n\omega}dm \right\rvert
&\leq& \sum_{i=1}^{I(k)}  \left\lvert \int_X P_\omega^{(n)}f  g_{i}^k dm\right\rvert  1_{F_i^k}(\sigma^n\omega)\nonumber\\
&\leq& I(k) \max _{1\leq i \leq I(k)}  \left\lvert \int_X P_\omega^{(n)}f  g_{i}^k dm\right\rvert.\nonumber
\end{eqnarray}
On the other hand, by the choice of $\omega$,
for any $1\leq i\leq I(k)$
one can find a positive integer $N_{i}=N_{i}(f, \omega, k )$ such that if $n\geq N_{i}$, then
\[
\left\lvert \int_X P^{(n)}_\omega f g_i^k dm \right\rvert<\frac{\epsilon }{I(k)}.
\]
Thus, for any $n\geq\max_{1\leq i\leq I(k)}N_i$,
\begin{align*}
\notag\left\vert  \int _X  P^{(n)}_{\omega}f g_{\sigma ^n \omega} dm \right\vert 
&\leq \left\vert  \int _X  P^{(n)}_{\omega}f g_{\sigma ^n \omega}^k dm \right\vert 
+ \left\Vert g_{\sigma ^n \omega}  -  g_{\sigma ^n \omega}^k \right\Vert _{L^\infty(X)}  \left\Vert P^{(n)}_{\omega}f \right\Vert _{L^1(X)}\\
&< (1+ \Vert f\Vert _{L^1(X)}  )\epsilon  .
\end{align*}
Since $\epsilon >0$ is arbitrary, we conclude
 that $(P,\sigma)$ is prior  mixing for inhomogeneous  observables in $L^\infty(\Omega,L^\infty(X,m))$.
\end{proof}

\subsection{Counterexamples}\label{ss:ce}

We give an example exhibiting prior  mixing for homogeneous observables but not  for inhomogeneous observables in  $B(\Omega, L^\infty (X,m))$. Let $T:X\to X$ be a measurably bijective  map (up to zero  $m$-measure sets) preserving $m$ such that the Perron-Frobenius operator $L_T$ associated with $T$ is mixing (note that $L_T 1_X =1_X$ due to the invariance of $m$ and  recall \eqref{eq:0912}). Note that the baker map is well-known example as such map $T$.
Assume that  there is a $\mathbb P$-positive measure set $\Omega _0$ such that the forward orbit of $\omega\in\Omega_0$ is measurable but not finite (e.g.~$\Omega =[0,1]$,  $\mathbb P$ is the Lebesgue measure on $\Omega$ and $\sigma$ is the tent map),
and that $P_\omega =L_T$ for all $\omega \in \Omega _0$. 
By construction, this Markov operator cocycle $(P, \sigma)$  is prior mixing for homogeneous observables. 

\begin{theorem}\label{cthm}
The Markov operator cocycle $(P, \sigma)$ given above is not prior mixing for inhomogeneous observables in  $B(\Omega, L^\infty (X,m))$.
\end{theorem}
\begin{proof}
We  first note that  the negation of  prior mixing for inhomogeneous  observables in $B(\Omega, L^\infty (X,m))$ is that
there is a measurable set $\Gamma \subset\Omega$ with $\mathbb{P}(\Gamma )>0$ such that for any $\omega\in \Gamma $, there exist $f=f^\omega$ in $L^1_0(X,m)$ and a bounded measurable map $g=g^\omega : \Omega \to L^\infty (X,m) : \tilde \omega \mapsto g^\omega _{\tilde \omega}$  satisfying 
$$
\lim_{n\to\infty}\int_X P_\omega^{(n)} f g_{\sigma^n\omega}dm
=\lim_{n\to\infty}\int_X P_\omega^{(n)} f^\omega g^\omega_{\sigma^n\omega}dm\neq 0.
$$
We emphasize that the observable $g=g^\omega$ may depend on $\omega\in \Gamma $. 

Let  $\Gamma=\Omega _0$ and fix $\omega \in \Gamma$.
Fix a measurable set $A$ with $m(A)= 1/2$.
Let $f=1_A- 1_{X\setminus A}$. 
Define $g =g^\omega : \Omega \to L^\infty (X,m)$  by
$$
g_{\tilde{\omega}}:=\begin{cases}
L_T^n 1_A & (\text{when $\tilde{\omega}=\sigma^n\omega$})\\
0 & {\rm (otherwise)}.
\end{cases}
$$
Then, by construction, $f\in L^1_0(X,m)$ and  $g : \Omega \to L^\infty (X,m)$ is a bounded and measurable map.
Furthermore, since $T$ is bijective,
for every $n\in \mathbb N$, 
\[
L^n_T1_A \cdot L_T^n 1_{X\backslash A} =0 \quad \text{$m$-almost everywhere}
\]
(note that 
$L^n_T1_B =1_{  T^n (B)}$ for any measurable set $B$).
Therefore, for every $n\in \mathbb N$
\begin{eqnarray}
\int_X P_\omega^{(n)} f g_{\sigma^n\omega}dm
&=&
\int_X L_T^n \left(1_A - 1_{X\setminus A}\right) L_T^n 1_Adm\nonumber\\
&=&\int_X \left( L_T ^n 1_A \right)^2dm =m(T^n(A)) =\frac{1}{2} >0.\nonumber
\end{eqnarray}
In conclusion, $(P,\sigma)$ is not prior  mixing for inhomogeneous observables in $B(\Omega, L^\infty$ $(X,m))$.
\end{proof}

\subsection{Skew-product transformations}

In this subsection, we show that our definitions of mixing for Markov operator cocycles naturally lead to the conventional mixing property for skew-product transformations.

Recall that $(X,\mathcal{A},m)$ and $(\Omega, \mathcal{F},\mathbb{P})$ are probability spaces, and $\sigma:\Omega\to\Omega$ is a $\mathbb{P}$-preserving transformation. 
We further assume that $\sigma$ is invertible and {\it mixing}. Let $(P,\sigma)$ be a Markov operator cocycle induced by the Perron-Frobenius operator corresponding to a non-singular transformation $T_\omega:X\to X$ for $\mathbb{P}$-almost every $\omega\in\Omega$. Assume that there is an invariant density map
$h : \Omega \to D(X,m)$ of $(P,\sigma )$
and define a measurable family of measures $\{\mu _\omega\}_{\omega \in \Omega}$ by $\mu_\omega(A)=\int_Ah_\omega dm$ for $A\in\mathcal{A}$, so that we have $(T_\omega )_*\mu _\omega = \mu _{\sigma \omega}$ due to \eqref{eq:0913b}.

Consider the skew-product transformation $\Theta:\Omega\times X \to \Omega\times X$ defined by $\Theta(\omega, x)=(\sigma\omega, T_\omega x)$ with the measure $\nu$ on $\Omega\times X$,
$$
\nu(A)=\int_{\Omega}\mu_\omega(A_\omega)d\mathbb{P}(\omega)\quad\text{for $A\in\mathcal{F}\otimes\mathcal{A}$},
$$
where $A_\omega:=\{x\in X: (\omega,x)\in A\}$ denotes the $\omega$-section.
Then, $(\Omega\times X, \mathcal{F}\otimes \mathcal{A},\nu)$ becomes a probability space, and $\nu$ is an invariant measure for $\Theta$, namely the Perron-Frobenius operator $L_\Theta$ corresponding to $\Theta$ with respect to $\nu$ satisfies $L_\Theta1_{\Omega\times X}=1_{\Omega\times X}$ $\nu$-almost everywhere.

\begin{theorem}\label{thm:0920}
If $(P,\sigma)$ is prior  mixing for inhomogeneous observables in $B(\Omega , L^\infty$ $(X,m))$, then $\Theta$ is mixing, that is, for any $A,B\in \mathcal{F}\otimes \mathcal{A}$,
\begin{eqnarray}\label{eq:12}
\lim_{n\to\infty}\nu(\Theta^{-n}A\cap B)=\nu(A)\nu(B).\label{mixingSP}
\end{eqnarray}
\end{theorem}

\begin{proof}

Let $1_{B_\omega}/\mu_\omega(B_\omega)\in D(X,\mu_\omega)$ so that $1_{B_\omega}h_\omega/\mu_\omega(B_\omega)\in D(X,m)$. Assuming that $(P,\sigma)$ is prior  mixing for inhomogeneous observables in $B(\Omega , L^\infty (X,m))$, we then know
\begin{eqnarray}
\lim_{n\to\infty}\int_X \left(P_\omega^{(n)}\left(\frac{1_{B_\omega}h_\omega}{\mu_\omega(B_\omega)}\right)-h_{\sigma^n\omega}\right) 1_{A_{\sigma^n\omega}} dm =0. \nonumber
\end{eqnarray}
Let $\widehat{P}_\omega:L^1(X,\mu_\omega)\to L^1(X,\mu_{\sigma\omega})$ be the normalized Markov operator defined by
$$
\widehat{P}_\omega f(x)=\begin{cases}
\frac{P_\omega(f h_\omega)(x)}{h_{\sigma\omega}(x)} & (x\in X^{\sigma\omega})\\
0 & ({\rm otherwise}).
\end{cases}
$$
where $X^\omega:={\rm supp}h_{\omega}$. Note that the relation $\widehat{P}_\omega 1_{X^\omega}=1_{X^{\sigma\omega}}$ holds for almost every $\omega\in\Omega$. Then we have,
$$
\lim_{n\to\infty}\left(\int_X \widehat{P}^{(n)}_\omega 1_{B_\omega}\cdot  1_{A_{\sigma^n\omega}} d\mu_{\sigma^n\omega}-\int_X 1_{A_{\sigma^n\omega}}d\mu_{\sigma^n\omega}\int_X 1_{B_\omega}d\mu_{\omega}\right)=0.\nonumber
$$
The first term can be calculated as
\begin{eqnarray}
\int_X \widehat{P}^{(n)}_\omega 1_{B_\omega}\cdot  1_{A_{\sigma^n\omega}} d\mu_{\sigma^n\omega}
&=&\int_X (\widehat{P}^{(n)}_\omega 1_{B_\omega})\cdot h_{\sigma^n\omega}\cdot  1_{A_{\sigma^n\omega}} dm\nonumber\\
&=&
\int_X P_\omega^{(n)} (1_{B_\omega}\cdot h_{\omega})\cdot  1_{A_{\sigma^n\omega}} dm\nonumber\\
&=&
\int_X 1_{B_\omega}\cdot h_\omega\cdot P^*_{\sigma\omega}\circ\cdots\circ P^*_{\sigma^{n-1}\omega} 1_{A_{\sigma^n\omega}} dm\nonumber\\
&=&
\int_X 1_{B_\omega}\cdot P^*_{\omega}\circ\cdots\circ P^*_{\sigma^{n-1}\omega} 1_{A_{\sigma^n\omega}} d\mu_{\omega}
\nonumber
\end{eqnarray}
Thus,
\begin{eqnarray}
\lim_{n\to\infty}\left(\int_X 1_{B_\omega}\cdot P^*_{\omega}\circ\cdots\circ P^*_{\sigma^{n-1}\omega} 1_{A_{\sigma^n\omega}} d\mu_{\omega}-\int_X 1_{A_{\sigma^n\omega}}d\mu_{\sigma^n\omega}\int_X 1_{B_\omega}d\mu_{\omega}\right)=0,\nonumber
\end{eqnarray}
where $P^*_\omega:L^\infty(X,m)\to L^\infty(X,m)$ is the Koopman operator with respect to $T_\omega$. Note that this implies the following natural mixing property for the random dynamical system $\{T_\omega\}_{\omega\in\Omega}$,
\begin{eqnarray}
\lim_{n\to\infty}\left(\mu_{\omega}\left(T^{(-n)}_{\omega} A_{\sigma^n\omega} \cap B_\omega\right)-\mu_{\sigma^n\omega}(A_{\sigma^n\omega})\mu_\omega(B_\omega)\right)=0,
\end{eqnarray}
where $T^{(-n)}_{\omega}=T^{-1}_{\omega}\circ\cdots\circ T^{-1}_{\sigma^{n-1}\omega}$.
Since $\mathbb{P}$ is a probability measure, by the Lebesgue dominated convergence theorem, 
\begin{eqnarray}
\lim_{n\to\infty}\left(\int_\Omega\mu_{\omega}\left(T^{(-n)}_{\omega}  A_{\sigma^n\omega} \cap B_\omega\right)d\mathbb{P}(\omega)-\int_\Omega\mu_{\sigma^n\omega}(A_{\sigma^n\omega})\mu_\omega(B_\omega)d\mathbb{P}(\omega)\right)=0.
\end{eqnarray}
By using $\displaystyle\Theta^{-n}(A)=\bigcup_{\omega\in\Omega}(\sigma^{-n}\omega, T^{-1}_{\sigma^{-n}\omega}\circ\cdots\circ T^{-1}_{\sigma^{-1}\omega} A_\omega)$ 
and $\displaystyle B=\bigcup_{\omega\in\Omega} (\omega,B_\omega)=\bigcup_{\omega\in\Omega} (\sigma^{-n}\omega, B_{\sigma^{-n}\omega})$, we have
\begin{eqnarray}
\nu(\Theta^{-n}A\cap B)&=&\int_{\Omega}\mu_{\sigma^{-n}\omega}(T^{-1}_{\sigma^{-n}\omega}\circ\cdots\circ T^{-1}_{\sigma^{-1}\omega} A_\omega \cap B_{\sigma^{-n}\omega})d\mathbb{P}(\omega)\nonumber\\
&=&\int_{\Omega}\mu_\omega \left(T^{(-n)}_{\omega} A_{\sigma^n\omega} \cap B_\omega\right)d\mathbb{P}(\omega)\nonumber
\end{eqnarray}

On the other hand, since $\sigma$ is mixing, invertible and $\mathbb{P}$-preserving,
\begin{eqnarray}
\lim_{n\to\infty}\int_{\Omega}\mu_\omega(A_\omega)d\mathbb{P} \int_{\Omega}\mu_\omega(B_\omega)d\mathbb{P}(\omega)
&=&
\lim_{n\to\infty}\int_\Omega  \mu_\omega(A_\omega)L_\sigma^n(\mu_\omega(B_\omega))d\mathbb{P}(\omega)\nonumber\\
&=&\int_{\Omega}\mu_\omega(A_\omega)d\mathbb{P} \int_{\Omega}\mu_\omega(B_\omega)d\mathbb{P}(\omega)\nonumber\\
&=&\nu(A)\nu(B)\nonumber
\end{eqnarray}
where $L_\sigma:L^1(\Omega,\mathbb{P})\to L^1(\Omega,\mathbb{P})$ is the Perron-Frobenius operator of $\sigma$.
Therefore we obtain 
$\nu(\Theta^{-n}A\cap B)\to\nu(A)\nu(B)$ as $n\to\infty$ for $A, B\in \mathcal{F}\otimes \mathcal{A}$.
\end{proof}

\begin{remark}\label{rm:0903c}
The converse of Theorem \ref{thm:0920} is not true in general due to the example given in Subsection \ref{ss:ce}. 
Indeed, let $\Theta$ be the direct product of two baker maps with $(\Omega \times X, \nu )  = ([0,1]^4, \mathrm{Leb} _{[0,1]^4})$. 
Then, the Markov operator cocycle $(P,\sigma )$ induced by $\Theta$ is not is prior  mixing for inhomogeneous observables in $B(\Omega , L^\infty (X,m))$ due to Theorem \ref{cthm}.
On the other hand, $\Theta$ is mixing because the backer map is mixing and the direct product of two same mixing systems is also mixing (cf.~\cite{W}). 
\end{remark}

\begin{remark}\label{rm:0903d}

In the case of prior mixing for homogeneous observables, as in the proof of Theorem \ref{thm:0920}, we can derive the  convergence
\begin{eqnarray}\label{eq:15}
\nu(\Theta^{-n}(F_1\times A_1)\cap (F_2\times A_2))\to\nu(F_1\times A_1)\nu(F_2\times A_2) \quad (n\to\infty)\label{mixingSP-2}
\end{eqnarray}
for any $F_1, F_2 \in\mathcal{F}$ and $A_1, A_2 \in \mathcal{A}$.
On the other hand,  \eqref{eq:15} implies the conventional mixing    of $\Theta$ 
 by a standard  approximation of measurable sets in $\mathcal F\otimes \mathcal A$ by finite union of direct product sets.\footnote{
Fix $A, B \in \mathcal F\otimes \mathcal A$ and $\epsilon >0$.
Then, one can find 
$\{F_{1}^{j} \}_{ j=1}^ J, \{ F_{2}^{j}\}_{ j=1}^ J\subset \mathcal F$ and $\{ A_{1}^{j}\}_{ j=1}^ J, \{ A_{2}^{j}\}_{ j=1}^ J\subset \mathcal A$  with $J\in \mathbb N$
such that both $ \{ A_1^j \times F_1^j \} _{j=1}^J$ and $ \{ A_2^j \times F_2^j \} _{j=1}^J$ are pairwise disjoint and 
both $\nu (A\setminus  (\cup _{j=1}^{J} A_1^j \times F_1^j)) $ and $\nu (A\setminus  (\cup _{j=1}^{J} A_2^j \times F_2^j)) $ are bounded by $\epsilon$.
By taking finer partition if necessary, one can assume that $\{F_{1}^{j} \}_{ j=1}^ J$  is also  pairwise disjoint. 
Since $\theta $ is invertible, $\{\Theta ^{-n} (F_{1}^{j} \times A_1^j) \}_{ j=1}^ J$ is again pairwise disjoint for each $n\geq 0$. 
Hence, 
$
\nu (\Theta ^{-n} A\cap B)-\nu (A)\nu (B) 
$ is $\epsilon$-close to
\[
\sum _{j=1}^J\left(\nu (\Theta ^{-n} (F_{1}^{j} \times A_1^j)\cap (F_{2}^{j} \times A_2^j))-\nu (F_{1}^{j} \times A_1^j)\nu (F_{2}^{j} \times A_2^j) \right),
\]
whose absolute value is smaller than $\epsilon$ for any sufficiently large $n$ by \eqref{eq:15}.
Since $\epsilon$ is arbitrary, this implies \eqref{eq:12}, that is, the mixing of $\Theta$.
}
Furthermore, recall   that  $\Theta$ is mixing if and only if for any $\tilde f \in L^1(\Omega \times X, \nu )$ and $\tilde g \in L^\infty (\Omega \times X, \nu )$
\begin{equation}\label{eq:15b}
  \int _{\Omega \times X} \tilde f  \tilde g \circ \Theta ^nd\nu \to \int _{\Omega \times X} \tilde f  d\nu   \int _{\Omega \times X} \tilde g d\nu  \quad (n\to\infty ).
\end{equation}
Therefore, 
if $\Theta$ is mixing, then for any $f \in L^1 _0(X, m)$, $g\in L^\infty (X,m)$ and $\rho \in L^\infty (\Omega , \mathbb P)$, by applying \eqref{eq:15b} to 
\[
\tilde f(\omega , x) = 
\begin{cases}
\frac{f(x)}{h_\omega (x)} \quad &(x\in \mathrm{supp}(h_\omega ))\\
0 \quad & (x\not\in \mathrm{supp}(h_\omega ))
\end{cases} \quad \text{and}
 \quad \tilde g(\omega ,x) =\rho (\omega )g(x), 
 \]
 it follows from \eqref{eq:0913} and the invariance of $\mathbb P$ for $\theta$ that 
\[
 \int_\Omega \rho (\omega ) \left(\int _{X} P_{\sigma ^{-n}\omega }^{(n)} f g dm \right) d\mathbb P(\omega ) 
   \to 0
     \quad (n\to\infty ).
\]
Since $\rho$ is arbitrary, this immediately implies the prior mixing of $(P,\sigma )$.

In conclusion, the prior mixing of $(P,\sigma )$ for homogeneous observables, \eqref{eq:15} and the mixing of $\Theta$ are equivalent, and thus,  
due to the relationship summarized in  Figure \ref{fig:FC}, 
the prior/posterior mixing for (in)homogeneous observables
in every   class considered in this paper  implies the conventional mixing of $\Theta$.  
\end{remark}

By Theorems \ref{Thm:mixing} and \ref{thm:0917b}  together with Remark \ref{rm:0903}, 
 prior mixing for homogeneous observables is equivalent to  posterior mixing for inhomogeneous observables in $L^\infty(\Omega , L^\infty (X,m))$, which is  equivalent to  posterior mixing for inhomogeneous observables in $B(\Omega , L^\infty (X,m))$ by definition (refer to Remark \ref{rm:0903b}).
Therefore, by Remarks  \ref{rm:0903c} and \ref{rm:0903d} we obtain the following corollary.
\begin{corollary}\label{cor:0903}
Let $(P,\sigma )$ be the Markov operator cocycle given in Remark \ref{rm:0903c}.
Then $(P,\sigma )$ is posterior mixing for inhomogeneous observables in $B(\Omega , L^\infty (X,m))$ but  not prior mixing for inhomogeneous observables in $B(\Omega , L^\infty (X,m))$. 
\end{corollary}

\begin{remark}
As we will see in Section \ref{s:asymp}, when $(P, \sigma )$ is an \emph{asymptotically periodic}  Markov operator cocycle, then the posterior mixing for inhomogeneous observables in $B(\Omega , L^\infty (X,m))$ is equivalent to the prior mixing for inhomogeneous observables in $B(\Omega , L^\infty (X,m))$.
This is contrastive to Corollary \ref{cor:0903}. 
On the other hand, the backer map, being the fiber dynamics of the counterexample in Corollary \ref{cor:0903}, is a well-known example whose Perron-Frobenius  operator is not  asymptotically periodic. 
\end{remark}

\subsection{Problems}
We finally propose a related problem.
Our definitions of mixing were given in terms of the decay of the correlation between $P^{(n)}_\omega f$ and $g$ (or $g_{\sigma^n\omega}$), and it is of great importance to evaluate the \emph{rate} of decay, as seen in the previous works \cites{BBR,BKS,DFGV,Gundlach1996}. Thus, we pose the following problem:
\begin{problem}
Investigate the relationship between decay rates of correlations for each type of mixing in Definition \ref{eme}.
\end{problem}

All results introduced in Remark \ref{B} established not only mixing property but also exponential mixing (for expanding or hyperbolic maps).
As mentioned there,
these results include both prior and posterior  mixing for both homogeneous and inhomogeneous observables.
We also remark that Froyland et al.~recently developed a multiplicative ergodic theorem for semi-invertible operator cocycles in \cites{FLQ,GQ}, which
 enabled one to consider the ``quasi-compactness'' of transfer  operator cocycles in terms of Lyapunov exponents and played the key role in the establishment of   exponential mixing (and its consequences such as several limit theorems) for random expanding  or hyperbolic dynamical systems  in \cites{DFGV,DFGV2}.

\section{Exactness}\label{s:exact}

As a characterization of exactness which is well-known for one non-singular transformation (see \cite{Aa}),
we have the generalization of Lin's theorem \cite{L} as follows. For each $\omega\in\Omega$, $P_{\omega}^*$ denotes the adjoint operator of $P_{\omega}$   defined by
\begin{align*}
\int_XP_{\omega}fgdm = \int_XfP_{\omega}^*gdm
\end{align*}
for $f\in L^1(X,m)$ and $g\in L^{\infty}(X,m)$, and we will use the notation
\begin{align*}
P_{\omega}^{(n)*} = P_{\omega}^*\circ P_{\sigma\omega}^*\circ\dots\circ P_{\sigma^{n-1}\omega}^*
\end{align*}
for $\omega\in\Omega$ and $n\ge 1$.

\begin{theorem}\label{exactness}
Let $(P,\sigma)$ be a Markov operator cocycle and $S=\{g\in L^{\infty}(X,m) : \left\Vert g \right\Vert_{L^{\infty}}\le 1\}$ the unit ball in $L^{\infty}(X,m)$.
Then the following are equivalent for each $\omega\in\Omega$.
 \begin{enumerate}
 \item $f\in L^1(X,m)$ satisfies $\left\lVert P_{\omega}^{(n)}f \right\rVert_{L^1(X)} \to 0$ as $n\to\infty;$
 \item $f\in L^1(X,m)$ satisfies $\int_X f g dm = 0$ for any $g\in\bigcap_{n\ge1} P_{\omega}^{(n)*}S$.
 \end{enumerate}
Consequently, 
$(P,\sigma)$ is exact
if and only if
$\bigcap_{n\ge1} P_{\omega}^{(n)*}S=\{c 1_X : c\in \mathbb{R}\}$ for $\mathbb{P}$-almost every $\omega\in\Omega$.
\end{theorem}

\begin{proof}
First of all, notice that for any $\omega\in\Omega$, $P_{\omega}^{*}S\subset S$ and $P_{\omega}^{(n)*}=P_{\omega}^*\circ P_{\sigma\omega}^*\circ\dots\circ P_{\sigma^{n-1}\omega}^*$ enable us to have the decreasing sequence in $L^{\infty}(X,m)$:
 \begin{align*}
 S \supset P_{\omega}^{*}S \supset P_{\omega}^{(2)*}S \supset \dots \supset \bigcap_{n\ge1} P_{\omega}^{(n)*}S.
 \end{align*}
 
Now we assume (1) is true.
Then for each $g\in\bigcap_{n\ge1} P_{\omega}^{(n)*}S$, there is a sequence $\{g_n\}_n\subset S$ so that $P_{\omega}^{(n)*}g_n=g$ and for $f$ in the condition (1),
 \begin{align*}
 \int_X f g dm = \int_X f P_{\omega}^{(n)*} g_ndm = \int_X P_{\omega}^{(n)} f g_ndm \le \left\lVert P_{\omega}^{(n)}f \right\rVert_{L^1(X)}\to 0
 \end{align*}
as $n\to\infty$.
Thus (2) is valid.

Next, suppose that  (2) holds.
By the Banach-Alaoglu theorem and continuity of $P_{\omega}^*$ on the weak-* topology in $L^{\infty}(X,m)$, $S$ is compact in weak-* and so is $P_{\omega}^{(n)*}S$.
For $f$ in the condition (2), taking $g_n=\operatorname*{sgn}\left(P_{\omega}^{(n)}f\right)\in S$ where $\operatorname*{sgn}(\phi)=1$ on $\{\phi\ge0\}$ and $\operatorname*{sgn}(\phi)=-1$ otherwise, we have
 \begin{align*}
 \left\lVert P_{\omega}^{(n)}f \right\rVert_{L^1(X)} = \int_X P_{\omega}^{(n)}f g_n dm = \int_X f P_{\omega}^{(n)*}g_n dm.
 \end{align*}
Let $g$ be an accumulation point of $\left\{P_{\omega}^{(n)*}g_n\right\}_n$ which belongs to $\bigcap_{n\ge1} P_{\omega}^{(n)*}S$.
Then we have $\int_X fgdm=0$ by assumption (2) and for some subsequence $\{n_i\}_i\subset\mathbb{N}$, we have
 \begin{align*}
 \lim_{i\to\infty} \left\lVert P_{\omega}^{(n_i)}f \right\rVert_{L^1(X)} = \lim_{i\to\infty} \int_X f P_{\omega}^{(n_i)*}g_{n_i} dm = \int_X fg dm = 0.
 \end{align*}
Since $P_{\omega}$ is Markov, $\left\lVert P_{\omega}^{(n)}f \right\rVert_{L^1(X)} \le \left\lVert P_{\omega}^{(n_i)}f \right\rVert_{L^1(X)}$ for $n\ge n_i$.
Therefore we have the condition (1).

Finally, considering the case when $f\in L^1_0(X,m)$, we have the equivalent condition for exactness of $(P,\sigma)$ and the proof is completed.
\end{proof}

As an immediate corollary of Theorem \ref{exactness}, we have:

\begin{corollary}\label{cor:32}
If a Markov operator cocycle $(P,\sigma)$ is derived from non-singular transformations $T_{\omega}$, that is, each $P_{\omega}$ is the Perron-Frobenius operator associated to $T_{\omega}$.
Then  $(P,\omega)$ is exact if and only if for $\mathbb{P}$-almost every $\omega\in\Omega$,
 \begin{align*}
 \bigcap_{n\ge1}\left( T_{\omega}^{(n)} \right)^{-1} \mathcal{A} = \{\emptyset,X\} \pmod{m}.
 \end{align*}
\end{corollary}

\begin{proof}
Since $P_{\omega}^*$ is the Koopman operator of $T_{\omega}$, characteristic functions are mapped to characteristic functions.
Thus we can consider $P_{\omega}^{(n)*}$ on $\{g\in L^{\infty}(X,m) : \lVert g\rVert_{L^{\infty}(X)}=1\}$ and we prove the corollary. 
\end{proof}

\section{Asymptotic periodicity}\label{s:asymp}

In the arguments of conventional Markov operators, it is known that mixing and exactness are equivalent properties in the asymptotically periodic class \cite{LM}. In this section, we consider a similar result to the conventional one for our definitions of mixing and exactness for Markov operator cocycles under the following definition of asymptotic periodicity, which is studied in \cite{NN}.
Moreover, in the sequel of the section, we introduce the relation between the asymptotic periodicity and exactness from the viewpoint of the existence of an invariant density.

\begin{definition}[Asymptotic periodicity]\label{ap}
 A Markov operator cocycle $(P, \sigma )$ is said to be {\it asymptotically periodic}  
 if there exist an integer $r$,  finite collections $\{ \lambda_i\} _{i=1}^r \subset B \left(\Omega , (L^1(X,m))^{\prime}\right)$ and 
 $\{ \varphi_i\} _{i=1}^r \subset B \left(\Omega , D(X,m)\right)$ satisfying that $\{\varphi_i ^\omega \}_{i=1}^r$ have mutually disjoint supports for $\mathbb{P}$-almost every $\omega\in\Omega$, and there exists a permutation $\rho_\omega$ of $\{1,\dots,r\}$ such that 
 \begin{align}\label{asydeco}
 P_\omega \varphi_i^{\omega}=\varphi_{\rho_\omega(i)}^{\sigma\omega} \quad \text{and} \quad \lim_{n\to\infty}\left\lVert P_\omega^{(n)}\left( f-\sum_{i=1}^r\lambda_i^\omega(f)\varphi_{i} ^{\omega }\right) \right\rVert_{L^1(X)}=0 
\end{align}
for every $ f\in L^1(X, m)$, $1\leq i\leq r$ and $\mathbb P$-almost every $\omega \in \Omega$, 
where $\lambda _i^\omega =\lambda _i(\omega )$, $\varphi_i^\omega =\varphi_i(\omega)$.

Furthermore, if in addition $r=1$, then $(P, \sigma )$ is said to be {\it asymptotically stable}.  
\end{definition}

Note that when $(P,\sigma)$ is asymptotically periodic,
\begin{align*}
h_{\omega}=\frac{1}{r}\sum_{i=1}^r\varphi_i^{\omega}
\end{align*}
becomes an invariant density for $(P,\sigma)$.

For an asymptotically periodic single Markov operator, exactness and mixing coincide with $r=1$ for the representation of asymptotic periodicity (see Theorem 5.5.2 and 5.5.3 in \cite{LM}).
The following theorem and proposition are Markov operator cocycles version of them.

\begin{theorem}\label{X}
Let $(P,\sigma)$ be an asymptotically periodic Markov operator cocycle.
Then the followings are equivalent.
 \begin{enumerate}
 \item $(P,\sigma)$ is exact;
 \item $(P,\sigma)$ is prior mixing for inhomogeneous observables in $B(\Omega,L^\infty(X))$;
 \item $(P,\sigma)$ is posterior mixing for inhomogeneous observables in $B(\Omega,L^\infty(X))$;
 \item $(P,\sigma)$ is asymptotically stable.
 \end{enumerate}
\end{theorem}

\begin{proof}
(1) $\Rightarrow$ (2) $\Rightarrow$ (3):
Obvious.

(3) $\Rightarrow$ (4):
Suppose $(P,\sigma)$ is posterior mixing for inhomogeneous observables in $B(\Omega,L^\infty(X,m))$ and $r>1$ (recall that $r$ is the period of the  asymptotically periodic Markov operator cocycle $(P,\sigma)$ given in Definition \ref{ap}).
By asymptotic periodicity of $(P,\sigma)$, we have an invariant density $h_{\omega}=\frac{1}{r}\sum_{i=1}^r\varphi_i^{\omega}$.
Set $g_{\omega}=1_{\operatorname*{supp}\varphi_1^{\omega}}$ and write
 \begin{align*}
 \int_X P_{\omega}^{(n)} \left(h_{\omega}-\varphi_1^{\omega}\right) \cdot g_{\sigma^n\omega} dm
 &= \int_{\operatorname*{supp}\varphi_1^{\sigma^n\omega}} \left(h_{\sigma^n\omega} - \varphi_{\rho_\omega^n(1)}^{\sigma^n\omega}\right) dm\\
 &= \frac{1}{r} - \int_{\operatorname*{supp}\varphi_1^{\sigma^n\omega}} \varphi_{\rho_\omega^n(1)}^{\sigma^n\omega} dm\\
 &=
  \begin{cases}
  \frac{1}{r} & (\rho_\omega^n(1)\neq 1) \\
  \frac{1}{r}-1 & (\rho_\omega^n(1)=1)
  \end{cases}
 \end{align*}
for each $n\ge0$.
This contradicts posterior mixing for inhomogeneous observables in $B(\Omega,L^\infty(X))$ of $(P,\sigma)$.

(4) $\Rightarrow$ (1):
We assume $(P,\sigma)$ is asymptotically periodic with $r=1$.
That is, for any $f\in L^1(X,m)$, $\left\lVert P^{(n)}_{\omega}(f-\lambda^{\omega}(f)\varphi^{\omega})\right\rVert_{L^1(X)}\to0$ as $n\to\infty$.
Since $P_{\omega}$ is Markov and $\varphi_{\omega}\in D(X,m)$ for each $\omega\in\Omega$,
for any $f\in D(X,m)$ we have $\lambda^{\omega}(f)=1$.
Thus $(P,\sigma)$ is exact by Remark \ref{B}.
\end{proof}

The following proposition reveals the relationship between two kinds of mixing and exactness.
Namely, in the setting of asymptotically periodic systems mixing for homogeneous observables, mixing for inhomogeneous measurable observables and exactness are equivalent under certain topological assumption of $\Omega$.

\begin{proposition}\label{XX}
Let $(P,\sigma)$ be an asymptotically periodic Markov operator cocycle.
Suppose $\sigma$ preserves a regular probability measure $\mathbb{P}$ on a metric space $(\Omega,\mathcal{F},\mathbb{P})$ with metric $d_{\Omega}$ and for $\mathbb{P}$-almost every $\omega\in\Omega$, each component of invariant densities $\varphi_i^{\omega}$ belongs to $C(\Omega,L^\infty(X,m))$.
Then the condition that $(P,\sigma)$ is prior mixing for homogeneous observables is necessary and sufficient for each condition in Theorem \ref{X}.
\end{proposition}

\begin{proof}
Necessity:
obvious.

Sufficiency:
we show $r=1$ in the representation of asymptotic periodicity.
Assume $r>1$ contrarily.
By our assumption, for $\mathbb{P}$-almost every $\omega\in\Omega$ and any $\epsilon>0$ there exists $\delta>0$ such that $d_{\Omega}(\omega, \omega ')<\delta$ implies $\left\lVert \varphi_i^{\omega} - \varphi_i^{\omega '} \right\rVert_{L^1(X)} < \epsilon$ for $i=1,\dots,r$.
Also, we have that for $i \neq j$, if $d_{\Omega}(\omega, \omega')<\delta$ then
 \begin{align*}
 \int_{\operatorname*{supp}\varphi_j^{\omega}} \varphi_i^{\omega'} dm \le \left\lVert \varphi_i^{\omega'} - \varphi_i^{\omega} \right\rVert_{L^1(X)} + \int_{\operatorname*{supp}\varphi_j^{\omega}} \varphi_i^{\omega} dm < \epsilon.
 \end{align*}
Set a $\delta$-ball $B_{\delta}(\omega)$ centered at a given $\omega\in\Omega$ which is of positive measure since $\mathbb{P}$ is regular.
Poincar\'{e}'s recurrence theorem tells us that $\sigma^n\omega$ visits $B_{\delta}(\omega)$ infinitely many times and let $\{n_k\}_k$ satisfy $\sigma^{n_k}\omega\in B_{\delta}(\omega)$.
Then for an invariant density $h_{\omega}=\frac{1}{r}\sum_{i=1}^r\varphi_i^{\omega}$ and for $A=\operatorname*{supp}\varphi_1^{\omega}$, if $\omega'$ satisfies $d_{\Omega}(\omega, \omega')<\delta$,
 \begin{align*}
 \left\lvert \int_A h_{\omega'} dm \right\rvert &=
 \frac{1}{r} \left\lvert \int_A \left( \varphi_1^{\omega'} - \varphi_1^{\omega} + \varphi_1^{\omega} \right)dm + \sum_{i\neq 1}\int_A \varphi_i^{\omega'}dm \right\rvert
 \end{align*}
and we have
 \begin{align*}
 \frac{1}{r} - \epsilon \le \left\lvert \int_A h_{\omega'} dm \right\rvert \le \frac{1}{r} + \epsilon.
 \end{align*}
Therefore, taking $0<\epsilon<\frac{1}{2r}$ we have
 \begin{align*}
 \left\lvert\int_AP^{(n_k)}_{\omega}\left(\varphi_1^{\omega}-h_{\omega}\right)dm\right\rvert
 &=\left\lvert\int_A\left(\varphi_{\rho_\omega^{n_k}(1)}^{\sigma^{n_k}\omega}-h_{\sigma^{n_k}\omega}\right)dm\right\rvert\\
 &\ge
  \begin{cases}
  \frac{1}{r} - 2\epsilon &(\rho_\omega^{n_k}(1)\neq1)\\
  (1 - 2\epsilon)-\frac{1}{r} &(\rho_\omega^{n_k}(1)=1)
  \end{cases}\\
 &>0.
 \end{align*}
This contradicts prior mixing of $(P,\sigma)$ for homogeneous observables.
\end{proof}

In order to relate asymptotic periodicity and exactness together with the existence of an invariant density, we give the definition of quasi-constrictiveness.

\begin{definition}
Let $(P,\sigma)$ be a Markov operator cocycle.
Then $(P,\sigma)$ is called {\it quasi-constrictive} if for any $\epsilon>0$ there exists $\delta>0$ such that for any $E\in\mathcal{A}$ with $m(E)<\delta$, it holds that
 \begin{align*}
 \operatorname*{lim\ sup}_{n\to\infty}\operatorname*{ess\ sup}_{\omega\in\Omega}\int_EP_{\omega}^{(n)}fdm<\epsilon\quad \text{for any }f\in D(X,m).
 \end{align*}
\end{definition}

For the sake of convenience, for an asymptotically periodic Markov operator cocycle $(P,\sigma)$ and each component of the invariant density $\varphi_i^{\omega}$,
we denote a measurable map $\omega \mapsto P_{\omega}^{(k)}\mid_{\operatorname*{supp}\varphi_i^{\omega}}$ by $P^{(k)}\mid_{\operatorname*{supp}\varphi_i}$ for $k\in\mathbb{N}$ and $i=1,\dots,r$.
In the following two propositions, we can see that (i): asymptotic periodicity of $(P,\sigma)$ is equivalent to (ii): the existence of an invariant density and exactness of $(P^{(k)}\mid_{\operatorname*{supp}\varphi_i},\sigma^k)$ for some $k\in\mathbb{N}$ and all $i=1,\dots,r$.

\begin{proposition}\label{XXX}
Let $(P,\sigma)$ be a Markov operator cocycle such that $P$ is strongly continuous i.e., $\omega \mapsto P_{\omega}f$ is continuous for each $f \in L^1(X,m)$.
Suppose $\Omega$ is compact, $(P,\sigma)$ has an invariant density $h_{\omega}$ and $\{\mu_{\omega}\}_{\omega\in\Omega}$ is uniformly absolutely continuous with respect to $m$ where  $d\mu_{\omega}=h_{\omega}dm$.
If $(P^{(k)},\sigma^k)$ is exact for some $k\in\mathbb{N}$, then $(P,\sigma)$ is quasi-constrictive.
\end{proposition}

\begin{proof}
By the assumption, $(P,\sigma)$ admits an invariant density $\{h_{\omega}\}_{\omega}$ and there exists $k\in\mathbb{N}$ such that, for any $\psi\in L^1_0(X,m)$ and $\mathbb{P}$-almost every $\omega$, we have $\left\lVert P^{(kn)}_{\omega}\psi\right\rVert_{L^1(X)}\to0$ as $n\to\infty$.
We show $(P,\sigma)$ is quasi-constrictive.
For any $N$ sufficiently large, $f\in D(X,m)$ and $E\in\mathcal{A}$,
 \begin{align*}
 &\sup_{n>N}\operatorname*{ess\ sup}_{\omega\in\Omega}\int_EP_{\omega}^{(n)}fdm
 \\
 =&\sup_{n_0k\ge N}\operatorname*{ess\ sup}_{\omega\in\Omega}\left\{\int_EP_{\omega}^{(n_0k)}fdm,\int_EP_{\omega}^{(n_0k+1)}fdm,\dots,\int_EP_{\omega}^{((n_0+1)k-1)}fdm\right\}.
 \end{align*}
For $j\in\{0,\dots, k-1\}$, we have
 \begin{align*}
 \int_EP_{\omega}^{(n_0k+j)}fdm
 &=\int_EP_{\sigma^j\omega}^{(n_0k)}P_{\omega}^{(j)}fdm
 \\
 &\le \left\Vert P_{\sigma^j\omega}^{(n_0k)}\left(P_{\omega}^{(j)}f-h_{\sigma^j\omega}\right)\right\Vert_{L^1(X)}+\int_EP_{\sigma^j\omega}^{(n_0k)}h_{\sigma^j\omega}dm
 \\
 &= \left\Vert P_{\sigma^j\omega}^{(n_0k)}\left(P_{\omega}^{(j)}f-h_{\sigma^j\omega}\right)\right\Vert_{L^1(X)}+\int_Eh_{\sigma^{n_0k+j}\omega}dm.
 \end{align*}
Thus, since $P_{\omega}^{(j)}f-h_{\sigma^j\omega}\in L^1_0(X,m)$ and $\Omega$ is compact, we have
 \begin{eqnarray}
&& \operatorname*{lim\ sup}_{n\to\infty}\operatorname*{ess\ sup}_{\omega\in\Omega}\int_EP_{\omega}^{(n)}fdm \nonumber\\
 &\leq& \operatorname*{lim\ sup}_{n_0\to\infty} \operatorname*{ess\ sup}_{\omega\in\Omega} \max_{0\le j \le k-1} \left\Vert P_{\sigma^j\omega}^{(n_0k)}\left(P_{\omega}^{(j)}f-h_{\sigma^j\omega}\right)\right\Vert_{L^1(X)} \nonumber
 \\
 &&\qquad + \operatorname*{lim\ sup}_{n_0\to\infty} \operatorname*{ess\ sup}_{\omega\in\Omega} \max_{0 \le j \le k-1} \mu_{\sigma^{n_0k+j}\omega}(E) \nonumber
 \\
 &\leq& \operatorname*{ess\ sup}_{\omega\in\Omega} \mu_{\omega}(E). \label{eqA}
 \end{eqnarray}
Indeed, if $\displaystyle \operatorname*{lim\ sup}_{n_0\to\infty} \operatorname*{ess\ sup}_{\omega\in\Omega} \max_{0\le j \le k-1} \left\Vert P_{\sigma^j\omega}^{(n_0k)}\left(P_{\omega}^{(j)}f-h_{\sigma^j\omega}\right)\right\Vert_{L^1(X)}\neq 0$,
there exist $\psi \in L^1_0(X,m)$, $\epsilon_0>0$, $\{ \omega_l \}_l \subset \Omega$ and $\{ n_l \}_l \subset \mathbb{N}$ such that $\left\Vert P_{\omega_l}^{(n_lk)} \psi \right\Vert_{L^1(X)} \ge \epsilon_0$.
Compactness of $\Omega$ ensures that there exists further subsequences $\omega'_s = \omega_{l_s}$ and $n'_s = n_{l_s}$ such that $\omega'_s \to \bar{\omega}$ for some $\bar{\omega} \in \Omega$ as $s$ tends to $\infty$.
Now strong continuity of $P$ and exactness of $P^k$ imply that
 \begin{align*}
 \left\Vert P_{\omega'_s}^{(n'_sk)} \psi \right\Vert_{L^1(X)}
 \le \left\Vert P_{\omega'_s}^{(n'_sk)} \psi - P_{\bar{\omega}}^{(n'_sk)} \psi \right\Vert_{L^1(X)} +\left\Vert P_{\bar{\omega}}^{(n'_sk)} \psi \right\Vert_{L^1(X)}
 \to 0
 \end{align*}
as $s \to \infty$ and this leads contradiction.

Uniform absolute continuity of $\{\mu_{\omega}\}_{\omega}$ with respect to $m$ implies
that for each $\epsilon>0$ there exists $\delta>0$
such that if $m(E)<\delta$ then (\ref{eqA}) is less than $\epsilon$.
Therefore the desired result is obtained.
\end{proof}

\begin{proposition}\label{XXXX}
Let $(P,\sigma)$ be an asymptotically periodic Markov operator cocycle such that the permutation $\rho_{\omega}\equiv\rho$ in Definition \ref{ap} is constant $\mathbb{P}$-almost everywhere.
Then $(P,\sigma)$ admits an invariant density and there exists a natural number $k$ such that $(P^{(k)}\mid_{\operatorname*{supp}\varphi_i},\sigma^k)$ is exact for $i=1,\dots,r$.
\end{proposition}

\begin{proof}
Obviously $h_{\omega}=\frac{1}{r}\sum_{i=1}^r\varphi_i^{\omega}$ is an invariant density for $P_{\omega}$.

Now let $k$ be the smallest number satisfying $\rho^k=\mathrm{id}$.
Then, setting $A_i^{\omega}=\operatorname*{supp}\varphi_i^{\omega}$, $P^{(k)}_{\omega}\mid_{A_i^{\omega}}$ is a Markov operator from $L^1(A_i^{\omega}, m)$ into $L^1(A_i^{\sigma^k\omega}, m)$.
By representation of asymptotic periodicity of $P_{\omega}$, for any $f\in D(A_i^{\omega}, m)$ we have that
 \begin{align*}
 \lambda_j^{\omega}(f)=
  \begin{cases}
  1 &(j=i)\\
  0 &(j\neq i).
  \end{cases}
 \end{align*}
This implies that for any $f\in D(A_i^{\omega}, m)$,
 \begin{align*}
 \lim_{n\to \infty} \left\lVert P_{\omega}^{(nk)} \left( f - \varphi_i^{\omega} \right) \right\rVert_{L^1(X)} =0.
 \end{align*}
Therefore we conclude $(P^{(k)}\mid_{\operatorname*{supp}\varphi_i},\sigma^k)$ is exact by Remark \ref{B}.
\end{proof}

\begin{remark}
As is well known (see \cite{LM} and references therein), quasi-constrictive single Markov operators were shown to be asymptotically periodic.
In the forthcoming paper \cite{NN}, we will generalize their result for Markov operator cocycles.
As a consequence of Proposition \ref{XXX} and Proposition \ref{XXXX} together with the results in \cite{NN}, we the following when $\Omega$ is compact.
 \begin{enumerate}
 \item
 If $(P,\sigma)$ admits an invariant density bounded below and above and $(P^{(k)},\sigma^k)$ is exact for some $k\ge1$, then $(P,\sigma)$ is asymptotically periodic.
 \item
 Conversely, if $(P,\sigma)$ is asymptotically periodic, then $(P,\sigma)$ admits an invariant density and there exists $k\ge1$ such that $(P^{(k)}\mid_{\operatorname*{supp}\varphi_i},\sigma^k)$ is exact for $i=1,\dots,r$.
 \end{enumerate}
In particular,
$(P,\sigma)$ is asymptotically periodic with period 1 if and only if $(P,\sigma)$ admits a unique invariant density and is exact.
\end{remark}

\section*{Acknowledgement}

The authors would like to express our gratitude to Professor Michiko Yuri (Hokkaido university) for giving us constructive suggestions and  warm encouragement. 
This work was partially supported by JSPS KAKENHI Grant Number
19K14575.

\end{document}